\documentclass[leqno,12pt]{article} 
  \usepackage{epsfig}

\usepackage{amsmath, amssymb}
\usepackage{amsthm} 
\theoremstyle{plain} 
\newtheorem{theorem}{\indent\sc Theorem}[section]
\newtheorem{lemma}[theorem]{\indent\sc Lemma}
\newtheorem{conj}[theorem]{\indent\sc Conjecture}
\newtheorem{corollary}[theorem]{\indent\sc Corollary}
\newtheorem{prop}[theorem]{\indent\sc Proposition}
\newtheorem{claim}[theorem]{\indent\sc Claim}
\theoremstyle{definition} 
\newtheorem{definition}[theorem]{\indent\sc Definition}
\newtheorem{remark}[theorem]{\indent\sc Remark}
\newtheorem{example}[theorem]{\indent\sc Example}
\newtheorem{notation}[theorem]{\indent\sc Notation}
\newtheorem{assertion}[theorem]{\indent\sc Assertion}
\newcommand{\sL}{{\mathcal{L}}}

\newcommand{\sO}{{\mathcal{O}}}
\newcommand{\sC}{{\mathcal{C}}}
 \newcommand{\resp}{{\it resp.} }
\newcommand{\ie}{{\it i.e.} }
\newcommand{\eg}{{\it e.g.} }
\newcommand{\cf}{{\it cf.} }
\newcommand{\loccit}{{\it loc. cit.} }
\newcommand{\Q}{\mathbb{Q}}
\newcommand{\R}{\mathbb{R}}
\newcommand{\C}{\mathbb{C}}
  \newcommand{\Z}{\mathbb{Z}}
   \newcommand{\coim}{\operatorname{coim}}
\newcommand{\coker}{\operatorname{coker}}

\newcommand{\IM}{\operatorname{Im}}
\newcommand{\Coim}{\operatorname{Coim}}

\newcommand{\rk}{\operatorname{rk}}

\newcommand{\im}{\operatorname{im}}
 \newcommand{\Spec}{\operatorname{Spec}}
\newcommand{\un}{\mathbf{1}}

\newtheorem*{remark0}{\indent\sc Remark}
\makeatletter
\def\address#1#2{\begingroup
\noindent\parbox[t]{7.8cm}{%
\small{\scshape\ignorespaces#1}\par\vskip1ex
\noindent\small{\itshape E-mail address}%
\/: #2\par\vskip4ex}\hfill%
\endgroup}%
\makeatother
\title{\uppercase{On {nef} and semistable hermitian lattices, and their behaviour under tensor product}} 
\author{
  \small{\it heartily dedicated to Professor Daniel Bertrand on his sixtieth birthday} 
\bigskip \\
\textsc{Yves ANDR\'E} 
}
  \date{\small (revised version, 7 august 2010)}  
\begin{document}

\maketitle




 
 \begin{abstract} We study the behaviour of semistability under tensor product in various settings: vector bundles, euclidean and hermitian lattices (alias Humbert forms or Arakelov bundles), multifiltered vector spaces. 
 
One approach to show that
semistable vector bundles in characteristic zero are preserved by tensor product
is based on the notion of nef vector bundles. We revisit this approach and show how far it can be transferred to hermitian lattices. J.-B. Bost conjectured that semistable hermitian lattices are preserved by tensor product. Using properties of {nef} hermitian lattices, we establish an inequality which improves on earlier results in that direction. On the other hand, we show that, in contrast to {nef} vector bundles, {nef} hermitian lattices are not preserved by tensor product. 
 
We axiomatize our method in the general context of monoidal categories, and give an elementary proof of the fact semistable multifiltered vector spaces (which play a role in diophantine approximation) are preserved by tensor product. 


 \end{abstract}

\section*{Introduction} 
 \begin{sloppypar}

 ${\bf 0.1.}$   Notions of (semi)stability and slope filtrations have been introduced and developed in many different mathematical contexts, often independently, sometimes by analogy.
In \cite{A}, we have shown how all these slope filtrations are, beyond analogy, special instances of a general notion which obeys a very simple formalism.

In the present sequel to \cite{A}, using the general formalism only as a guiding thread, we revisit and exploit some of these concrete analogies, with emphasis on the case of euclidean (and hermitian) lattices.

\medskip  The theory of euclidean lattices has evolved in connection with crystallography, algebraic number theory and, more recently, cryptography and mathematical physics. Nevertheless, since Hermite's time, its main focus has remained, against the backcloth of classification problems, on the question of finding shortest or nearest vectors, or short nearly orthogonal bases, and on the related question of finding lattices with good ``packing" or ``kissing" properties.  
 Reduction theory aims at estimating the length of short vectors, and more generally the (co)volumes of small sublattices of lower ranks, of lattices of given rank and (co)volume, and at combining lower and upper bounds to get finiteness results.   
 
 A better grasp on lower bounds comes from the more recent part of reduction theory which deals with semistability and slope filtrations (heuristically, semistability means that the Minkowski successive minima are not far from each other, \cf \cite{Bor}).
   These notions were introduced by U. Stuhler \cite{St}, inspired by the analogy with semistability and slope filtrations for vector bundles on curves (Mumford, Harder-Narasimhan). They have been further developed in this spirit in the context of Arakelov geometry. They provide interesting finite partitions of the space of isometry classes of lattices \cite{Gr}\cite{C} and fundamental domains for the action of $SL_n(\Z)$ on the space of positive quadratic forms of rank $n$ and fixed discriminant.   Curiously, however, they do not seem to have attracted interest among ``classical" lattice-theorists. It is only very recently that an algorithm has been devised to compute slope filtrations \cite{Ki} (which has also helped to investigate the relation between semistability and Voronoi's classical notion of perfection, \loccit).
   
 \medskip 
  
 \noindent  ${\bf 0.2.}$ To be more specific, let $\bar E$ be an euclidean lattice, \ie a free abelian group $E$ of finite rank with an euclidean structure $\langle\; ,\;\rangle$ on the real vector space which its spans. The analog of the {\it degree} of a vector bundle is given by 
 $$ \widehat{\rm deg}\,\bar E = -{\log {\rm  vol}\, \bar E}.$$  It behaves additively in short exact sequences (the euclidean structure of the middle term inducing the euclidean structure of the other terms).
 If $\bar E \neq 0$, one defines the {\it slope} by
 $$\mu(\bar E) =  \frac{ \widehat{\rm deg}\,\bar E}{{\rm rk}\, \bar E}\,\cdot$$  
 One introduces the supremum $\mu_{\max}(\bar E)$ of the slopes of all nonzero sublattices (of any rank) of $\bar E$, and one says that $\bar E$ is {\it semistable} if $\mu(\bar E)=\mu_{\max}(\bar E)$. 
 In that case, the dual $\bar E^\vee$ is also semistable, of opposite slope.  Any euclidean lattice is, in a unique way, a successive extension of semistable ones with increasing slopes.

\medskip For instance, any integral lattice $\bar E$  (\ie such that the euclidean product takes integral values on $\bar E$) satisfies $\mu_{\max}(\bar E)\leq 0$, and it is unimodular if and only if $\mu(\bar E)=0$; in that case, $\mu_{\max}(\bar E)$ is also $0$. Therefore any unimodular integral lattice is semistable of slope $0$ (examples: root lattice ${\mathbb E}_8$,  Leech lattice). Indecomposable root lattices are semistable \cite{Ki}.
  
 If one associates to any point $\tau$ of the upper half plane the plane lattice generated by $1, \tau$,  
 the (contractible!) region of semistable lattices of rank two then corresponds to the closed exterior of the Ford circles in the strip $0<\Im \tau \leq 1$ (\cf \cite{C}).
   
 \medskip 

 \noindent  ${\bf 0.3.}$    The study of tensor products of euclidean lattices has been undertaken by Y. Kitaoka in a series of papers (\cf \eg \cite[Ch. 7]{K}), notably from the viewpoint of shortest vector problems. 
 
For any finite family of nonzero euclidean lattices $\bar E_i$, one has the formula
  \begin{equation}\label{0.3.1}\mu(\otimes \bar E_i)= \sum \mu(\bar E_i).\end{equation}

  \smallskip In the context of vector bundles on a projective smooth curve, in characteristic zero, it is a well-known (but non-trivial) that the tensor product of two semistable objects is semistable.
 
 J.-B. Bost has conjectured the same for euclidean lattices\footnote{\textsc{J.-B. Bost}, Hermitian vector bundles and stability, talk at Oberwolfach, 
 Algebraische Zahlentheorie, July 1997, quoted in \cite{Che}.}. Taking into account the additivity of $\widehat {\rm deg}$, this is equivalent to:
 
 \begin{conj}  For any finite family of nonzero euclidean lattices $\bar E_i$, one has    \begin{equation}\label{0.3.2}\mu_{\max}(\otimes \bar E_i)= \sum \mu_{\max}(\bar E_i).\end{equation}\end{conj}
 
This holds for instance if all $\bar E_i$ are integral unimodular, since $\otimes \bar E_i$ is also integral unimodular, hence semistable of slope $0$.
 
  \smallskip In spite of its elementary formulation, this conjecture seems challenging. There are a number of partial results about it in small rank, \cf \eg \cite{DSP}\cite{Za}. Note that the lower bound $\,\mu_{\max}(\otimes \bar E_i)\geq \sum \mu_{\max}(\bar E_i) \,$ follows from \eqref{0.3.1}. In \cite[3.37]{BK},  the following upper bound is proven
  
 \begin{prop}\label{prop02}   \begin{equation}\label{0.3.3}\mu_{\max}(\otimes \bar E_i)\leq \sum (\mu_{\max}(\bar E_i)+ \frac{1}{2}\log {\rm rk}\,\bar E_i).\end{equation}
 \end{prop}
 
 In the first section of this paper, we present a quick elementary proof of this inequality\footnote{the results of this text where first exposed in Paris in the conference in honour of D. Bertrand (march 2009). At the time, I was not yet aware that the reference \cite{BK} contained a proof of Proposition \ref{prop02}.}.
 
\medskip

  \noindent  ${\bf 0.4.}$  In the second section, we turn to vector bundles on a projective smooth curve $S$ over a field $k$ of characteristic zero. Recall that the slope of a nonzero vector bundle $E$ on $S$ is $\mu(E)= \deg E/ {\rm rk}\, E$, and that $E$ is semistable if all nonzero subbundles $F$ have lower or equal slope.
  
  Tensor products of semistable vector bundles are semistable (in the terminology of \cite{A}, the Harder-Narasimhan slope filtration is $\otimes$-multiplicative); equivalently:
  
  \begin{theorem}  For any finite family of nonzero vector bundles $ E_i$ on $S$, one has    \begin{equation}\label{0.4.1}\mu_{\max}(\otimes   E_i)= \sum \mu_{\max}(  E_i).\end{equation}
  \end{theorem}

   There are three known proofs. 
 One proof uses the ``transcendental" description by M. Narasimhan and C. Seshadri of semistable vector bundles in terms of unitary representations of the fundamental group \cite{NS}.
 
 Another one uses geometric invariant theory and Kempf filtrations \cite{RR}.
 
 A third one (\cf \cite{Maru}\cite{Mi}\cite{Laz}\cite{Br}) relies on the relation between semistability and numerical effectivity (a vector bundle is {\it {nef}} if its pull-back along any finite covering of $S$ has no quotient line bundle of negative degree), and more precisely on the well-known

   \begin{prop} A vector bundle of degree zero is {nef} if and only if it is semistable.
  The tensor product of {nef} vector bundles is {nef}.
  \end{prop}
  
 We give a simple version of this third proof, which does not even use the fact that the tensor product of {nef} bundles is {nef}.  Our argument works as well in the case of strongly semistable vector bundles in characteristic $p$.
 
   \medskip
    
  \noindent  ${\bf 0.5.}$  In the third section, we come back to the arithmetic situation, and examine {\it hermitian lattices} ${\bar E}$ over the ring of integers ${\frak o}_K$ of a number field $K$.  This generalization of euclidean lattices appears in Arakelov theory as (hermitian) vector bundles on the arithmetic curve $\bar S  = \Spec {\frak o}_K \cup V_\infty$ (where $V_\infty$ denotes the set of archimedean places of $K$).  In fact, they were already considered by P. Humbert in 1940, in the equivalent language of hermitian forms rather than lattices, and have been further studied in the spirit of classical lattice theory under the name ``Humbert forms" \cite{I}\cite{C4}. Curiously, however, these two trends seem to ignore each other.
  
 Taking appropriate products over $V_\infty$, one defines a variant of (co)volume for a hermitian lattice $\bar E$. One then introduces the invariants $\widehat{\rm deg}\,\bar E,\;\mu(\bar E)$  and $ \mu_{\max}(\bar E)$, and the notion of semistability, as above. 
    Bost's conjectural equality \eqref{0.3.2} actually concerns hermitian lattices, not just euclidean lattices.
   Therein we prove the following generalization of Proposition 0.2\footnote{J.-B. Bost has informed me that he has also proved this result, apparently by a different method, \cf  \textsc{J.-B. Bost}, Stability of Hermitian vector bundles over arithmetic curves and geometric invariant theory, talk at Chern Institute, Nankai, April 2007.}.   
  
 \begin{theorem}  For any finite family of nonzero hermitian lattices $\bar E_i$ on $\bar S$, one has
   \begin{equation}\label{0.4.1}\mu_{max}(\otimes \bar E_{i } ) \leq \sum_i\, (\mu_{max}(\bar E_i) + \frac{[K:\Q]}{2}   \log {\rm rk}\, \bar E_i ).\end{equation}\end{theorem}
 
 This improves on earlier results \cite[3.37]{BK}\cite{Che} (in \cite{BK}, an extra term involving the discriminant of $K$ appears, whereas H. Chen \cite{Che}, by an arithmetic elaboration of the method of \cite{RR}, finds a term $ {[K:\Q]}   \log {\rm rk}\, \bar E_i $ instead of $\frac{[K:\Q]}{2}   \log {\rm rk}\, \bar E_i $). But our proof is of less elementary nature: it relies on a difficult arithmetic analog of Kleiman's criterion proved by S. Zhang \cite{Z}. In fact, we import the notion of ``{nef}" in the context of hermitian lattices on $\bar S$, and try to follow systematically the proof which we have devised in the case of ordinary vector bundles, which uses the usual comparison between invariants of $E$ and invariants of $\sO_{\mathbb P(E)}(1)$.
 
  This comparison, for hermitian lattices, is precisely the place where the factor $\frac{[K:\Q]}{2}   \log r$ shows up (as a sharp upper bound for the Faltings height of $\mathbb P^{r-1}_K$), and one could not get rid of it in this place. 
 Indeed, in contrast to Proposition 0.4:
  
  \begin{prop} A {nef} hermitian lattice of degree zero is not necessarily semistable.
  The tensor product of two {nef} hermitian lattices of degree zero is not necessarily {nef}.
  \end{prop}
  
  This puts some limitation to the geometric-arithmetic analogy which is the leading thread of Arakelov theory\footnote{another such failure of the analogy: for vector bundles, semistability is an {\it open} condition, whereas it is a {\it closed} condition for hermitian lattices (as the above figure illustrates in the case of euclidean lattices of rank two).}. On the other hand, this also shows, in our opinion, that Bost's conjecture (if true) lies beyond this analogy.
  
  \medskip 
 \noindent  ${\bf 0.6.}$ After having declined the argument in three concrete contexts, its formalization in the most general categorical setting becomes transparent, and can be further concretised in other contexts. 
  
   \medskip 
 \noindent  ${\bf 0.7.}$ Vector bundles on a curve defined over a finite field and hermitian lattices can both be described as {\it adelic vector bundles}, \ie  finite-dimensional $K$-vector spaces endowed with a suitable collection of norms, \cf \cite{Gau}\cite{HoJS}. 
 
 We introduce the closely related notion of {\it generalized vector bundle}, and notions of slope and semistability for them, which allows to account as well for vector bundles on a curve defined over an arbitrary field $k$, and also for finite-dimensional $K$-vector spaces $M$ endowed with finitely many decreasing  filtrations  $F^{\geq .}_{\nu}$ (possibly defined over a finite separable extension $L/K$) \cite{F}\cite{Ra}.  
 
 Semistability of multifiltered spaces plays a role in the theory of diophantine approximations and in  $p$-adic Hodge theory. 
     Tensor products of multifiltered spaces are semistable (in the terminology of \cite{A}, the Faltings-Rapoport slope filtration is $\otimes$-multiplicative); equivalently:
  
  \begin{theorem}  For any finite family of nonzero multifiltered spaces $\bar M_i = ( M_i, F^{\geq .}_{i,\nu})$, one has    \begin{equation}\label{0.4.1}\mu_{\max}(\otimes   \bar M_i )= \sum \mu_{\max}(  \bar M_i ).\end{equation}
  \end{theorem}

   There are three known proofs. 
  In \cite{FW}, G. Faltings and G. W\"{u}stholz relate multifiltered spaces to vector bundles on curves, as follows (when ${\rm char}\,K=0$ and when the breaks of the filtration are rational)\footnote{let $S$ be a cyclic covering of ${\mathbb P}^1$, totally ramified above at least $n[L$:$K]$ branch points. To $\bar M$, they associate a vector bundle on $S$ of rank $\dim M$ and slope  $[L$:$K]\mu(\bar M)$, which is semistable if $\bar M$ is - and conversely, provided the degree of the covering $S/{\mathbb P}^1$ is large enough.  The result thus follows from the case of vector bundles (Theorem 0.3).}.
 
 Another proof sketched in \cite{F} (for $L=K$) uses a Rees module construction and a deformation argument due to G. Laffaille.
 
 A third one uses geometric invariant theory and Kempf filtrations \cite{To1}.

  \smallskip We give a completely elementary new proof of Theorem 0.7, valid in any characteristic, which is inspired by our quick proof of Proposition 0.2. This answers a question of G. Faltings \cite{F}.

 \section{A quick and elementary proof of Proposition 0.2} 
 
 \subsection{} Let us start with a couple of general remarks about euclidean lattices. First of all, they form a {\it category}, a fact which seems to be ostensibly ignored in the literature on euclidean lattices: morphism are linear maps of norm $\leq 1$ (we use here the operator norm, \ie the maximum of the norm of value of the map on the unit ball). Isomorphisms are isometries. 
 
 Of course this category is not (pre)additive. It has finite coproducts (orthogonal sums) but no finite products in general (the diagonal map $\Z\to \Z \perp \Z$ has norm $\sqrt 2 >1$). 
 Nevertheless, this category has kernels and cokernels, and subquotients behave nicely (in the terminology of \cite{A}, it is proto-abelian). One has the notion of short exact sequence  $0\to \bar E'\to \bar E\to \bar E''\to 0$:  namely, a short exact of abelian groups, the euclidean norm on $E'_\R$ (\resp $E''_\R$) being induced by (\resp  quotient of) the norm of $E_\R$.

Moreover, it is a symmetric monoidal category with respect to the natural tensor product $\otimes$. 

The dual of an euclidean lattice $\bar E = (E, \vert \vert \; \vert \vert_{E_\R})$ is  $\bar E^\vee= (E^\vee= Hom(E, \Z), \vert\vert\; \vert\vert_{E^\vee_\R})$. The dual is contravariant (a morphism and its transpose have the same norm). Note however that the standard evaluation map $\bar E\otimes  \bar E^\vee \to \Z$  has norm $\sqrt{ \rk E}$, hence is not a morphism of euclidean lattices if $\rk E>1$\footnote{in this respect, it was abusive to write in \cite[12.1]{A} that the monoidal category of euclidean lattices is ``rigid".}. 

Any morphism 
 $f: \Z \to \bar E_1\otimes \bar E_2$ (\ie any vector of norm $\leq 1$ in $E_1\otimes E_2$) gives rise to a morphism $f': \bar E_2^\vee\to \bar E_1$: indeed, in the canonical identification $E_{1,\R}\otimes E_{2,\R} \cong Hom_\R(E_{2,\R}^\vee, E_{1,\R})$, the norm of the left-hand side corresponds to the Hilbert-Schmidt norm on the right-hand side, which is $\geq$ the operator norm. Note that $f\mapsto f'$ is injective (and functorial in $E_1, E_2$), but not surjective in general. 
 
 Any euclidean lattice $\bar L$ of rank one is invertible with respect to $\otimes$, with inverse $\bar L^\vee$.
  
 \subsection{}  In order to prove Proposition \ref{prop02}, it is enough to take $i\in \{1,2\}$. On multiplying the euclidean norms by suitable constants, we may assume that the maximal volume of nonzero sublattices of   $\bar E_i$ is $1$, \ie $\mu_{\max}(E_i)\leq 0$. 
 
  Let $\bar E$ be a nonzero sublattice of $\bar E_{1 }\otimes \bar E_{2 }$. Let $r$ be its rank. It is enough to show that $\mu(\bar E)\leq   \frac{1}{2}\log r$, \ie that the volume of $\bar E$ is at least $r^{-r/2}$.

 Any euclidean sublattice $\bar L$ of $\bar E_{1 }\otimes \bar E_{2 }$ of rank one gives rise to a nonzero morphism    
   $ f':\, \bar E_{2}^\vee \to   L^\vee \otimes  \bar E_{1} $. By our normalization of $\bar E_i$, and by duality, 
   any quotient of $ \bar E_{2}^\vee$ (with quotient norm) has volume $\leq 1$; and any (not necessarily saturated) sublattice of $\bar L^\vee \otimes  \bar E_{1} $, with induced norm, has volume $\geq 1/{\rm vol}\, \bar L$. Factorizing the map $f'$ of norm $\leq 1$ through the quotient by its kernel, one gets that ${\rm vol}\, \bar L \geq 1$.
    
   Taking $\bar L\subset \bar E$, one gets that any nonzero vector in  $  \bar E$ has length $\geq 1$. 
     By Minkovski's theorem,  this implies that ${\rm vol}\, \bar E \geq 2^{-r}.v_r$ (where $v_r$ denotes the volume of the unit ball in $\mathbb R^r$). One concludes by noting that $2^{-r}.v_r\geq r^{-r/2}$, since the unit ball contains the hypercube of side $2r^{-1/2}$ centered at the origin. \qed
   
   \begin{remark} Instead of invoking Minkovski's theorem, one could simply bound $({\rm vol}\, \bar E)^{-2/r}$ from above by the Hermite constant  $\gamma_r$ (which is the supremum over euclidean lattices $\bar E$ of rank $r$ of the quantity $N(\bar E). {\rm vol}(\bar E)^{-2/r},$  where $N(\bar E) $ stands for the square of the length of a shortest vector).
   
   The bound $\mu(\bar E)\leq   \frac{1}{2}\log \gamma_r$ is significantly better than $\mu(\bar E)\leq   \frac{1}{2}\log r$ in small rank $r$, where $\gamma_r$ is explicitly known (this has been exploited in \cite{DSP}\cite{Za}). But this is not significative when $r\to \infty$, as $\log \gamma_r\sim \log r$.
     \end{remark}

\section{Tensor product of semistable vector bundles on a curve (proof of Theorem 0.3)}

\subsection{}  Let us first recall some basic facts about numerical effectivity (\cf \cite[ch. 6]{Laz}). Let $S$ be a projective smooth curve over an algebraically closed field $k$. Let $E$ be a vector bundle of finite rank $r$ on $S$. Recall that $E$ is said to be {\it {nef}}  if for any finite surjective morphism $S'\to S$, any  quotient line bundle $\,L\,$ of the pull-back $E' = E_{S'}$ has nonnegative degree. By normalization, it is enough to consider smooth curves $S'$.  

 It is clear that any quotient of a {nef} vector bundle $E$ on $S$ is {nef}.

\smallskip A pair $(S'/S, L)$ as above corresponds to a finite morphism $S'\to  {\mathbb P}(E)$ such that $L$ is the pull-back of $\sO_{{\mathbb P}(E)}(1)$. Therefore, $E$ is {nef} if and only if $\sO_{{\mathbb P}(E)}(1)$ is {nef} on ${\mathbb P}(E)$ in the sense that its inverse image on any curve $S'$ has nonnegative degree.

According to a fundamental result of S. Kleiman \cite{Kl} (which relies on the theory of ample line bundles), any {nef} line bundle $\sL$ on a projective variety $X$ of dimension $r$ satisfies  $ c_1(  \sL_{Y})^{\dim Y}\geq 0$ for any closed subvariety $Y$ of $X$, and in particular
 \begin{equation}\label{Kl} c_1(\sL)^{\cdot r}\geq 0  . \end{equation}
On the other hand, since $\dim S=1$, one has 
   \begin{equation}\label{eq1} \deg E  = c_1(\sO_{{\mathbb P}(E)}(1))^{\cdot r}, \end{equation} whence the well-known
   \begin{lemma}[Kleiman]\label{klei} If $E$ is {nef} on $S$, then $\, \deg E\geq 0 .\;\;\square$\end{lemma}

\subsection{}  If the degree of the covering $S'/S$ is $d$, then the degree of the pull-back $E'=E_{S'}$ is
\begin{equation}\label{muvb} \deg \, E'= d\cdot \deg \, E, \end{equation}
that is, $\mu(E') = d\mu(E)$, whence  $ \mu_{\max}( E') \geq  d \cdot \mu_{\max}(E) $ (where $\mu_{\max}$  denotes the maximum among the slopes of subbundles, or equivalently, among the slopes of coherent subsheaves). 

  {\it If $\,{\rm char}\, k=0$}, this is an equality:
 \begin{equation}\label{mumaxvb} \mu_{\max}( E')= d\cdot \mu_{\max}(E)_  , \end{equation}
as one sees by Galois descent of the (unique) subbundle  of $E'$ of maximal rank with maximal slope  (in fact, the Harder-Narasimhan filtration of $E'$ is the pull-back of the Harder-Narasimhan filtration of $E$).

\subsection{}  In the proof of Theorem 0.3, equation \eqref{mumaxvb} allows to replace $S$ by any finite covering (with $S'$ smooth). On the other hand it is enough to take $i\in \{1,2\}$ and to establish the upper bound for $\mu_{\max}(E_1\otimes E_2)$, and one may twist $E_i$ by any line bundle. Replacing $S$ by $S'$ finite over $S$ of degree divisible by the ranks $r_i$ of $E_i$ and twisting $E_i$ by suitable line bundles, we may assume that the maximal degree of coherent subsheaves of $E_i$ is $0$, for $i=1,2$ (\ie, that $\mu_{\max}(E_i)=0$). In particular $E_1^\vee$ and $E_2^\vee$ are {nef} (the trick of passing to a finite covering avoids the use of $\Q$-divisors).
  We then have to show that any nonzero subbundle  $ E$  of $  E_{1 }\otimes   E_{2 }$ has nonpositive degree. 
 
 Let $S'/S$ be any finite covering (with $S'$ smooth), and let  $E_i'$ denote the pull-back of $E_i$ on $S'$. Any line subbundle $ L$ of $  E'_{1 }\otimes E'_{2 }$ of rank one gives rise to a nonzero morphism    
   $ f':\, (E'_{2})^\vee \to   L^\vee \otimes   E'_{1} $. By our normalization of $  E_i$ and \eqref{mumaxvb}, and by duality, 
   any quotient of $( E'_{2})^\vee$  has nonnegative degree; and any coherent subsheaf of $ L^\vee \otimes    E'_{1} $  has degree $\leq   \deg L^\vee$.  Factorizing $f'$  through the quotient by its kernel, one gets that $L^\vee$ has nonnegative degree. 
   
   This shows that $( E_{1 }\otimes E_{2 })^\vee$ is {nef}, and so is its quotient $E^\vee$. It follows from  Lemma \ref{klei} that $\deg E\leq 0$. \qed

 \smallskip \begin{remark} It follows from the lemma that a vector bundle of degree $0$ is {nef} if and only if it is semistable. More generally, a vector bundle is {nef} if and only if all of its slopes (\ie the breaks of the Harder-Narasimhan filtration) are nonnegative (R. Hartshorne).
 
\smallskip \noindent   The above proof of Theorem 0.3 does not use the fact that ``{nef} $\otimes$ {nef} is {nef}; rather, in characteristic $0$, this fact may be viewed as a consequence of Theorem 0.3. This point will be important for our arithmetic paraphrase of this proof in the next section.
  \end{remark}

   \section{Tensor product of semistable hermitian lattices (proof of Theorem 0.5)} 
   
   \subsection{}  We will transfer as closely as possible the lines of the above proof in the arithmetic setting. 
Let us first mention that our categorical comments on euclidean lattices in subsection 1.1 extend verbatim to hermitian ${\frak o}_K$-lattices (for any number field $K$). 

A {\it hermitian ${\frak o}_K$-lattice} $\bar E $ is a projective ${\frak o}_K$-module $E$ of finite rank endowed, for each archimedean place $v$ of $K$, with a positive quadratic (\resp hermitian) form on  the real (\resp complex) vector space $E\otimes_{{\frak o}_K} K_v$ (we adopt the convention that the hermitian scalar product is left antilinear).

For any ${\lambda}\in \R$, and any hermitian lattice $\bar E$ of rank $r$, we denote by $\bar E\langle{{\lambda}}\rangle$ the hermitian lattice obtained from $\bar E$ by multiplying all norms by $\, e^{-{\lambda}/[K:\Q]}$.  
The 
alternate products ${\rm Alt}^p \bar E$ are the usual ones at the level of ${\frak o}_K$-lattices, with hermitian  products defined by the formula   $\langle v_1 \wedge \cdots \wedge v_p, w_1\wedge \cdots \wedge w_p\rangle = det (\langle v_i,  w_j\rangle)$, and  $\det \bar E =  {\rm Alt}^r \bar E$.  The natural map $ E^{\otimes p}\to  {\rm Alt}^p   E $   has norm $\sqrt{p!}$, hence is not a morphism of hermitian lattices if $p>1$.

The degree of a hermitian lattice of rank one $\bar L$ is 
   \begin{equation} \widehat{\deg}\, \bar L = \log \sharp (L/{\frak o}_K \ell)-  {\epsilon_v} \sum_{v\in V_\infty}\log \vert\vert \ell\vert\vert_v \end{equation}  where $\ell$ is any nonzero vector in $L$, and $\epsilon_v$ is $1$ or $2$ according to whether it is real or not.
 
 The (arithmetic) {\it degree} of a hermitian lattice $\bar E$ of any rank  is
    \begin{equation}    \widehat{\deg}\, \bar E =  \widehat{\deg}\,\det \bar E .\end{equation}
    It follows from this formula that
   $$\widehat {\deg}\, \bar E^\vee= -\widehat {\deg}\, \bar E,\;\;\;\;\widehat{\deg}\, \bar E\langle{{\lambda}}\rangle =  \widehat{\deg}\, \bar E+ r{\lambda},$$
      that the degree is additive with respect to short exact sequences, and that the associated slope function $\mu= \widehat {\deg}  / \rk $ is additive with respect to tensor products.

\subsection{} Let $X$ be an integral projective scheme of dimension $r$, flat over $S= \Spec {\frak o}_K$. One has the notion of {\it {nef}} $C^\infty$-hermitian line bundle on $X$:
        $\bar \sL$ is {nef} if the restriction of $\sL$ to any fiber of $X/S$ is {nef} (in the algebro-geometric sense), if $c_1(\bar \sL)$ is a semipositive current on $X(\C)$ (for any complex point of $S$), and if moreover $\hat c_1(\bar \sL_{S'}) > 0$ for any integral subscheme $S'$ of $X$ which is finite and flat over $S$, \cf \cite[\S 2]{Mo1} (and \cite[\S 1]{Z} for the notions of semipositive current and smooth hermitian line bundle on a singular complex variety). 
             
   According to a fundamental result of S. Zhang \cite{Z} (which relies on his theory of ample hermitian line bundles), one can replace the latter condition by: $\hat c_1(\bar \sL_{Y})^{\dim Y}\geq 0$ for any integral subscheme $Y$ of $X$ which is flat over $S$. In particular (\cite[Lemma 5.4]{Z})
\begin{equation}\label{Zh} \hat c_1(\bar\sL)^{\cdot r}\geq 0  . \end{equation}

   \smallskip Let $\bar E$ be a hermitian lattice of rank $r$, viewed as a hermitian vector bundle on $S$. We say that $\bar E$ is {\it {nef}} if for any finite extension $K'/K$,  any rank one quotient $\,\bar L\,$ of the pull-back $\bar E' = \bar E_{S'}$ (with  $S'= \Spec {\frak o}_{K'}$) has nonnegative (arithmetic) degree.  
      
  Since the restriction of  $\sO_{{\mathbb P}(E)}(1)$ to any fiber is certainly is {nef}, and $c_1(\bar  \sO_{{\mathbb P}(\bar E)}(1))$ is a semipositive current on ${\mathbb P}(E)(\C)$, one sees that the hermitian lattice $\bar E$ is {nef} if and only if the hermitian line bundle $\bar \sL =\sO_{{\mathbb P}(\bar E)}(1)$ on $X= {\mathbb P}(\bar E)$  is {nef}.
  
 \begin{remark}\label{r3.1} 1) In order to check that $\bar E$ is nef, it is enough to check that any rank one {\it free} quotient of $ \bar E_{S'}$ has nonnegative degree. Indeed, any rank one quotient $\,\bar L\,$ of  $ \bar E_{S'}$ becomes free after pulling back to  $ \Spec {\frak o}_{K''}$ for a suitable extension $K''/K'$ (\eg the Hilbert class field of $K'$). 
 
 \smallskip \noindent 2) The orthogonal sum of nef hermitian lattices is nef. Indeed let $\bar L$ be a rank one quotient of $\bar E'_1 \perp \bar E'_2 $. The restriction of the quotient morphism to $\bar E'_i$ is nonzero for $i=1$ or $2$ (say $i=1$). Let $\bar L'$ be its image. Then $\widehat {\deg}\, \bar L\geq \widehat {\deg}\, \bar L' \geq 0$ since $\bar E_1$ is nef.
 \end{remark}

\smallskip We now come to the point where the strict parallel with the geometric case breaks down: namely \eqref{eq1} is no longer true. In fact, the quantity 
  \begin{equation}\label{eq1.5}\hat c_1(\sO_{{\mathbb P}(\bar E)}(1))^{\cdot r}\geq 0
   \end{equation} 
    is by definition the (nonnegative) Faltings height of ${\mathbb P}(\bar E)$ in the sense \cite{BGS} up to a factor $[K:\Q]$, and one has the formula  
    (\loccit (4.1.4)):
   \begin{equation}\label{eq2}
       \widehat\deg \,\bar E =  \hat c_1(\sO_{{\mathbb P}(\bar E)}(1))^{\cdot r} - \frac{  [K:\Q]r}{2}   \sum_{m=2}^{r} \frac{1}{m}   \geq  \hat c_1(\sO_{{\mathbb P}(\bar E)}(1))^{\cdot r} -  \frac{ [K:\Q]}{2} r\log r, 
      \end{equation} 
(beware the notations: in \cite{BGS}, ${\mathbb P}(E^\vee)$ stands for what we denote by ${\mathbb P}(E)$ following A. Grothendieck; this explains the sign difference between \eqref{eq2} and the formula  \loccit).  Whence the 

  \begin{lemma}\label{zhan} If $\bar E$ is {nef} on $S$, then $\, \widehat\deg \, \bar E\geq -  \frac{  [K:\Q]}{2} r\log r .\;\;\;\;\;\;\square$\end{lemma}

\subsection{}  If the degree of the covering $S'/S$ is $d= [K':K]$, then the degree of the pull-back $\bar E'=\bar E_{S'}$ is
\begin{equation}\label{muvb} \widehat\deg \,\bar E'= d\cdot  \widehat\deg \,\bar E, \end{equation}
that is, $\mu(\bar E') = d\mu(\bar E),$ whence  $ \mu_{\max}(\bar E') \geq  d \cdot \mu_{\max}(\bar E) $ (where $\mu_{\max}$  denotes the maximum among the slopes of sulattices, or equivalently, among the slopes of saturated sublattices). In fact
  \begin{equation}\label{mumaxvb} \mu_{\max}(\bar E')= d\cdot \mu_{\max}(\bar E)_  , \end{equation}
as one sees by Galois descent of the (unique) sublattice  of $\bar E'$ of maximal rank with maximal slope  (in fact,  the Stuhler-Grayson slope filtration of $\bar E'$ is the pull-back of the Stuhler-Grayson slope filtration of $\bar E$). On the other hand
\begin{equation}\label{tw} \widehat\deg \,(\bar E\langle \lambda\rangle)'= \widehat\deg \,\bar E'\langle d\lambda\rangle .\end{equation}

\subsection{}  In the proof of Theorem 0.5, it is enough to take $i\in \{1,2\}$, and one may replace $\bar E_i$ by $\bar E_i\langle{{\lambda}_i}\rangle$ for any constant ${\lambda}_i$. We may thus assume that the maximal degree of hermitian sublattices of $\bar E_i$ is $0$, for $i=1,2$ (\ie that $\mu_{\max}(\bar E_i)=0$). In particular $\bar E_1^\vee$ and $\bar E_2^\vee$ are {nef}.
  It is then enough  to show that for any nonzero sublattice  $ \bar E\subset   \bar E_{1 }\otimes   \bar E_{2 }$ of rank $r$,  $\widehat \deg E \leq \frac{ [K:\Q]r}{2} \log r $. 
  
  The argument is strictly parallel to the one given above in the geometric case, except that Lemma \ref{zhan} replaces Lemma \ref{klei}. \qed

   \begin{remark}  Lemma \ref{zhan} also follows from the ``absolute Siegel lemma" \cite{RT}\cite{P}: for any $\epsilon >0$, there exists an extension $K'/K$ and rank one quotients   $\bar L_i, \,( i= 1, \ldots, r)$ of $\bar E_{S'}$ which span $E^\vee_{K'}$, with
      $$\frac{1}{ [K':K]}\sum_1^r\, \widehat\deg(\bar L_i) \leq  \widehat\deg (\bar E) +  \frac{[K:\Q]r}{2}\log r + \epsilon .$$ If $\bar E$ is {nef}, the left-hand side is nonnegative.
      
      \smallskip In connection with Lemma \ref{zhan}, let us also mention the following theorem of N. Hoffmann \cite{Ho}: for any semistable hermitian lattice $\bar E_1$ (for instance ${\frak o}_K$), there is a hermitian lattice $\bar E_2$ of rank $r$ such that $\bar E_1\otimes \bar E_2$ has no rank one quotient of   negative degree, and
      $$\widehat\deg (\bar E_1\otimes \bar E_2) <   - \frac{[K:\Q]r}{2}(  \log r  -\log 2e\pi) - \frac{r\log \mathfrak{d}_K}{2}.$$ 
      \end{remark}

      \subsection{Questions} $1)$ Is it true that the exterior and symmetric powers of a semistable hermitian lattice are semistable? 
      
  \smallskip\noindent $2)$   Is it true that the tensor product of polystable hermitian lattices (= orthogonal sum of stable hermitian lattices of the same slope) is polystable?
      
      Note: the geometric analog is true (\cf \cite[9.1.3]{A}, where this is proven in the much more general context of $\otimes$-multiplicative slope filtrations in quasi-tannakian categories). On the other hand, this is true for integral unimodular lattices (which are polystable of slope $0$, any unimodular sublattice of an integral lattice being an orthogonal summand, \cf \cite[ch. 1]{Mart}).

     \section{Counter-examples (proof of Proposition 0.6)}

 \subsection{} Let us show by an example that Lemma \ref{klei} does not hold in the arithmetic case: that is, a nef hermitian lattice may have negative degree. 
 
\smallskip Let $\mathbb A_2$ be the root lattice with Gram matrix  $  \, (\begin{smallmatrix} 2&1  \\  1&2  \end{smallmatrix}) \,$  in some basis $(e_1, e_2)$.
 
Let us fix ${{\lambda}} \in [\frac{1}{2}\log\frac{3}{2}, \, \frac{\log 3}{4} [$. Then   $\mathbb A_2{\langle{{\lambda}}\rangle}$ has degree
$$\widehat\deg\, (\mathbb A_2{\langle{{\lambda}}\rangle}) =  {2{{\lambda}}} - \frac{\log 3}{2}  \, \in [   \frac{\log 3}{2} -  {\log 2}  , \; 0[.$$ 
    Since the length of shortest vectors of  $\mathbb A_2{\langle{{\lambda}}\rangle}$ is $\sqrt{2}e^{-{{\lambda}}}$, any rank one sublattice has degree $\leq {{\lambda}} - \frac{\log 2}{2}  <   \mu(\mathbb A_2{\langle{{\lambda}}\rangle})$. In particular, $\mathbb A_2{\langle{{\lambda}}\rangle}$ is stable of negative degree. This also shows, by additivity of the degree and since ${\lambda} \geq \frac{1}{2}\log\frac{3}{2}$, that any quotient of rank one of $\mathbb A_2{\langle{{\lambda}}\rangle}$ has nonnegative degree. 
 
 Let us show that this remains true after any finite extension $K'/\Q$, so that $\mathbb A_2{\langle{{\lambda}}\rangle}$ is {nef} (taking into account Remark \ref{r3.1}$.1)$).
   
Let $\ell = ae_1 + be_2, \, a,b \in {\frak o}_{K'},$ be a nonzero vector in  $\mathbb A_2{\langle{{\lambda}}\rangle}_{  {\frak o}_{K'}}$. It is enough to show that the hermitian ${\frak o}_{K'}$-lattice spanned by ${\ell}$ has degree  $\leq   [{K'}:\Q]({{\lambda}} -\frac{\log 2}{2} )  $;  in other words, that
 $$\prod_\sigma\, \vert\vert \sigma(a) e_1 + \sigma(b) e_2\vert\vert^2  \geq (2e^{-{2{\lambda}}} )^{ [{K'}:\Q]}  $$ (product over the complex embeddings $\sigma$ of ${K'}$).
 
 One may assume $ab\neq 0$. Since the angle between $e_1$ and $e_2$ is $\pi/3$, one has $$\frac{e^{2{{\lambda}}}}{2}\vert\vert \sigma(a) e_1 + \sigma(b) e_2\vert\vert^2 =   \vert \overline{\sigma(a)}\vert^2 + \vert \sigma(b)  \vert^2 + Re (\overline{\sigma(a)} {\sigma(b)} )  $$
$$=  ( \vert \overline{\sigma(a)} \vert- \vert \sigma(b) \vert )^2 + 2  \vert \overline{\sigma(a)} {\sigma(b)}\vert + Re (\overline{\sigma(a)} {\sigma(b)} ) \geq \vert \overline{\sigma(a)} {\sigma(b)}\vert = \vert  {\sigma(a b)}\vert  .$$
 Since $a$ and $b$ are nonzero algebraic integers,  $\prod_\sigma\, \vert  {\sigma(a b)}\vert \geq 1$. 
 
 This finishes the proof that $\mathbb A_2{\langle{{\lambda}}\rangle}$ is {nef}.  

\smallskip By Remark \ref{r3.1}$.2)$, it follows that for any $\lambda'\geq 0$, $\mathbb A_2{\langle{{\lambda}}\rangle}\perp \Z{\langle{{\lambda'}}\rangle}$ is also nef. In particular {\it $\mathbb A_2{\langle{{\lambda}}\rangle}\perp \Z{\langle \frac{\log 3}{2}- {2\lambda}  \rangle}$ is nef of degree $0$, but not semistable} (it has a positive and a negative slope).

        \subsection{} Let us now show that ``{nef} $\otimes$ {nef} is not necessarily {nef}", in the arithmetic case.  

  \smallskip   Let   ${\bar E}$ be the unique indecomposable unimodular integral hermitian lattice of rank three over $\mathcal O_K = \Z[\omega],\; \omega= \frac{1+\sqrt{-7}}{2}$: in a suitable basis $(e_1,e_2, e_3)$, its hermitian Gram matrix (with entries $\langle e_i, e_j\rangle$) is
$$  \begin{pmatrix}  2& \omega &1\\ \bar\omega & 2&1\\1&1&2 \\   \end{pmatrix}   $$(\cf \cite[4.4]{C3}\cite[p. 415]{H}).

 Let us fix ${{\lambda}} \in \, ] \frac{\log 3}{2} ,\;  \log 3-\frac{2}{3}\log 2]$. 
   We shall show that
{\it  ${{\bar E}}{{\langle{-{\lambda}}\rangle}}$ is {nef} but ${{\bar E}}{{\langle{-{\lambda}}\rangle}}^{\otimes 2}$ is not. }

\medskip As a vector in ${{\bar E}}\otimes_{\frak{o}_K} {{\bar E}}^\vee $ (which can be identified to $  {{\bar E}}^{\otimes 2}$ since $\bar E$ is unimodular), the identity has length $\sqrt 3$. Dually, it gives rise to a rank one quotient of ${{\bar E}}{{\langle{-{\lambda}}\rangle}}^{\otimes 2}$ of degree  $-2{{\lambda}} +  {\log 3}  <0,$ hence ${{\bar E}}{{\langle{-{\lambda}}\rangle}}^{\otimes 2}$ is not {nef}.

\smallskip In order to prove that ${{\bar E}} {{\langle{-{\lambda}}\rangle}}$ is {nef}, let us show, dually, that any vector  $\ell  \in {\bar E}_{  {\frak o}_{K'}}$ spans a hermitian lattice of degree $\leq -{{\lambda}}[K' :K]$  (taking into account Remark \ref{r3.1}$.1)$); in other words (by definition of the degree), that 
 $$\prod_\sigma\, \vert\vert \sigma(\ell ) \vert\vert_\sigma^2  \geq e^{ {{\lambda}} [K' :K]} $$ (product over the complex embeddings $\sigma$ of ${K'}$ which induce identity on $K$).

 \medskip Let us first treat the case when $\ell $ belongs to the sublattice generated by $e_1$ and $e_2$.  We note that $e_1$ and 
 $\, e'_2 = -e_1+ \bar\omega e_2\,$
 form an orthogonal basis of  $Ke_1\oplus Ke_2$, and that  $\vert\vert e_1\vert\vert^2=\vert\vert e'_2\vert\vert^2=2$.  Moreover $\ell $ can be written
 $$ \ell  = \frac{\omega}{2}(ae_1+ be'_2),\;\; a,b\in {\frak o}_{K'}$$ and $$ \vert\vert \sigma(\ell ) \vert\vert_\sigma^2 =   \vert  \sigma(a)\vert^2  + \vert\sigma(b) \vert^2 .$$
 One thus has
 $$\prod_\sigma\, \vert\vert \sigma(\ell ) \vert\vert_\sigma^2  \geq  2^{[K':K]}  {\vert N_{K'/K}(ab)\vert},$$ which is  $   > e^{ {{\lambda}} [K':K]} $ if $ab\neq 0$ (since $\lambda < \log 2$).   If $ab=0$, then $\ell $ is in fact an integral multiple of  $e_1$ or $e'_2$ and one concludes as well.

 \medskip In order to treat the general case, let us first note that 
 $$ f_3= \omega^2 e_1 + \bar \omega^2 e_2 + 2 e_3 $$ is orthogonal to  $e_1,e_2$ and that
 $$ \vert\vert  f_3\vert\vert ^2 = 2\langle e_3, f_3\rangle = 2( \omega^2 +  \bar\omega^2 +4)= 2. $$ 
 Besides, in order to exploit the symmetry between $e_1$ and $e_2$ which appears in the Gram matrix, it is useful to introduce the numbers   $$\theta^\pm = \bar\omega \,(\pm\frac{\sqrt 2}{2}-1)$$ and the orthogonal vectors 
 $$f^{\pm}_1= e_1 + \theta^\pm e_2,\;\; f^{\pm}_1= \bar  \theta^\pm  e_1 + e_2. $$ 
Since $\vert\theta^\pm\vert^2= 3\mp 2\sqrt 2$, one has
 $$ \vert\vert  f_1^\pm\vert\vert ^2 = \vert\vert  f_2^\pm\vert\vert ^2 =   2\sqrt 2 (  \sqrt 2 \mp 1).$$
 
One may assume $K'\supset K(\sqrt 2)$ and divide the set of embeddings $\sigma$ into two parts (denoted by $ \Sigma^+$ and $\Sigma^-$ respectively): those which act as identity on $K(\sqrt 2)$, and the other ones. For $\sigma\in \Sigma^\pm$, we use the orthogonal basis $(f^\pm_1,f^\pm_2,f_3)$ of ${{\bar E}}\otimes_{{\frak o}_K}K'$. There exists $a,b,c\in {\frak o}_{K'}$ such that for any $\sigma\in \Sigma^\pm$,
 $$ \sigma(\ell ) = \frac{1}{2}(af^\pm_1+ bf^\pm_2 + cf_3 ), $$ and one has
 $$ \vert\vert \sigma(\ell ) \vert\vert_\sigma^2 =  \frac{1}{2}( \vert  \sigma(a)\vert^2  +  \sqrt 2(\sqrt 2 \mp 1)( \vert\sigma(b) \vert^2 + \vert\sigma(c)\vert^2)).$$

Taking into account the inequality of arithmetic and geometric means $A+B+C\geq 3(ABC)^{1/3}$, one thus gets  $$\prod_\sigma\, \vert\vert \sigma(\ell ) \vert\vert_\sigma^2   \geq   (3. 2^{-2/3})^{ [K':K]}   \vert N_{K'/K}(  abc )\vert^{2/3}   \geq  e^{{{\lambda}} [K':K]} \vert N_{K'/K}(  abc )\vert^{2/3},  $$ which is 
greater or equal to $  e^{{{\lambda}}[K':K]}$ if $abc\neq 0$.

\smallskip It remains to deal with the case when  $c\neq 0$ but $a$ or $b$ is zero. One remarks that in those cases, one has in fact  $\omega \ell  \in {\frak o}_{K'} f^\pm_2 + {\frak o}_{K'}f_3$ or else $\bar\omega \ell  \in {\frak o}_{K'} f^\pm_1 + {\frak o}_{K'}f_3$, and one concludes as well.

This finishes the proof that ${{\bar E}}{{\langle{-{\lambda}}\rangle}}$ is {nef}.  

\smallskip Since $\bar E$ is integral unimodular, $ ({{\bar E}}{{\langle{-{\lambda}}\rangle}})^\vee$ is semistable of positive slope, and thus is also nef. By Remark \ref{r3.1}$.2)$, it follows that   {\it   ${{\bar E}}{{\langle{-{\lambda}}\rangle}}\perp ({{\bar E}}{{\langle{-{\lambda}}\rangle}})^\vee$ is {nef} of degree $0$, but its tensor square is not}.

  \smallskip \begin{remark}   In the geometric  case, the standard way of proving that ``{nef} $\otimes$ {nef} is {nef}" is by showing first that large symmetric powers of a {nef} bundle $E$ are {nef},  taking advantage of the formula  $H^0({\mathbb P}(E), \sO_{{\mathbb P}(E)}(n))\cong S^n   E $.
  
  Let us see what breaks down in the arithmetic case. Let $\bar E$ be a {nef} hermitian ${\frak o}_K$-lattice of rank $r$ and let us consider  $\bar \sL =    \sO_{{\mathbb P}(\bar E)}(1)$ over $X  ={\mathbb P}(E)$. According to S. Zhang \cite[Cor. 5.7]{Z}, for $n>>0$, $H^0(X,  \bar \sL^{\otimes n } )$ is spanned by its sections of supnorm $\leq 1$. But $H^0(X,  \bar\sL^{\otimes n})\cong (S^n \bar E)\langle{\rho_n}\rangle$, where $\rho_n = \frac{1}{2} \log     (\begin{smallmatrix}r -1+n\\ n  \end{smallmatrix} ) $. Hence  $(S^n  \bar E  ){\langle{{\lambda}}\rangle}$ appears as a quotient of the hermitian lattice $ ( \sO _K\langle  {{\lambda}} - \rho_n\rangle)^{(\begin{smallmatrix}r -1+n\\ n  \end{smallmatrix} )}  $, which is {nef} provided ${\lambda} \geq \rho_n = O(\log n)$. The latter constraint does not allow to apply this efficiently to $ \bar E = \bar E_{1 } \perp \bar E_{2 }  \perp  {\frak o}_K^{n-2}$, for instance. 
  \end{remark}
   
   \subsection{} Finally, we show (in analogy with the geometric case) that a hermitian lattice $\bar E$ whose rank one quotients are of nonnegative degrees is not necessarily nef, and does not necessarily satisfy $\hat c_1(\sO_{{\mathbb P}(\bar E)}(1))^{\cdot r}\geq 0$. 
  
 For $p= 5, 13$ or else $37$, the Hilbert class field of $K= \Q(\sqrt{-p})$ is $K'=K(\sqrt{-1})$. One has $\frak{o}_{K'} =   \frak{o}_K \frac{1+\sqrt p}{2}\oplus \frak{o}_K \sqrt{-1} $, which contains   $   \frak{o}_K  \oplus \frak{o}_K \sqrt{-1} =    \frak{o}_K  (1+\sqrt p) \oplus \frak{o}_K \sqrt{-1}$ as a subgroup of index $4$.

Let us make $\frak{o}_{K'} $ into a hermitian $\frak{o}_K$-lattice by means of the hermitian form $$\langle x, y\rangle = \frac{1}{2}tr_{{K'} /K}\, \bar xy.$$ Its sublattice $ \frak{o}_K  \oplus \frak{o}_K \sqrt{-1}$ is then the orthogonal sum of two copies of the unit lattice $\frak{o}_K $. Since it has index $4$, it follows that $$\widehat{\deg}\, \frak{o}_{K'}  = 2\log 2.$$ Our example $\bar E$ will be the dual of $\frak{o}_{K'} $. By the first equality $   \widehat\deg \,\bar E =  \hat c_1(\sO_{{\mathbb P}(\bar E)}(1))^{\cdot 2} - \frac{  [K:\Q] }{2}    $ in \eqref{eq2}, one has 
$$ \hat c_1(\sO_{{\mathbb P}(\bar E)}(1))^{\cdot 2} = 1-2\log 2 <0.$$
By Zhang's theorem, $\bar E$ is not nef. This can also be viewed directly as follows: since $K'$ is unramified over $K$, the right ${\frak o}_{K'}$-algebra ${\frak o}_{K'}\otimes_{{\frak o}_{K}}{\frak o}_{K'}$ is isomorphic  ${\frak o}_{K'} \oplus  {\frak o}_{K'}$, the two factors being permuted by $Gal(K'/K)$ (which is an isometry group of the lattice $\bar E$). This provides an orthogonal decomposition
$$ {\bar E_{{\frak o}_{K'}} } \cong {{\frak o}_{K'}}\langle -\log 2 \rangle  \perp  {{\frak o}_{K'}}\langle -\log 2\rangle $$
  where ${\frak o}_{K'}$ stands for the unit ${\frak o}_{K'}$-hermitian lattice (it follows that $\bar E$ is semistable). 

\smallskip
Let us show, on the other hand, that  any rank one $\frak{o}_K $-sublattice $\bar L$ of $\frak{o}_{K'}= \bar E^\vee$ has nonpositive degree. We note that $\bar L^\flat = \bar L\cap  ( \frak{o}_K  \perp \frak{o}_K \sqrt{-1})$ has index $\iota= 1,2$ or $4$ in $L$. If $\iota =1$, $\widehat{\deg}\,\bar L\leq 0$ since  $  \frak{o}_K  \perp \frak{o}_K $ is semistable of degree $0$.   If $\iota = 2$, $\bar L^\flat$ is not equal to either factor in $ \frak{o}_K  \perp \frak{o}_K \sqrt{-1}  $ (since those factors are saturated in $\frak{o}_{K'}$), and we have to show that $\widehat{\deg}\,\bar L^\flat\leq -\log 2$. Since $\bar L^\flat  $ becomes free after tensoring by $\frak{o}_{K'}$, it is enough to show that for any $\ell = (a,b)\in   \frak{o}_{K'}  \perp \frak{o}_{K'}   $ with $a,b\neq 0$,  
 $$\prod_\sigma\, \vert\vert \sigma(\ell ) \vert\vert_\sigma^2  \geq e^{ 2({\log 2}) [K' :K]}= 16 $$ (product over all complex embeddings of $K'$). But $\prod_\sigma\, \vert\vert \sigma(\ell ) \vert\vert_\sigma^2  = \prod_\sigma\, (\vert  \sigma(a) \vert^2 + \vert  \sigma(b) \vert^2 ) \geq 2^{[K':\Q]} \prod_\sigma\,  \vert  \sigma(ab) \vert   = 2^4 \vert N_{K'/\Q}(ab) \vert \geq 16.$ 
 
 Finally, if $\iota = 4$, then  $ L=  {\frak o}_{K}  \frac{1+\sqrt p}{2}\subset  {\frak o}_{K'}$, which has degree $- \log \frac{p+1}{4}<0$.

\section{Formalization}   
\subsection{}
 In any category with a zero object $0$, \ie an object which is both initial and terminal (such an object is unique up to unique isomorphism), there is a unique zero morphism between any two objects (a morphism which factors through $0$), and for any morphism $\, f$, one has the notions of kernel $\,\ker f$, cokernel $\,\coker f$, coimage $\,\coim f\,$ (cokernel of the kernel) and image $\,\im f\,$ (kernel of the cokernel) of  $f$.

\smallskip In a category with kernels and cokernels (\ie with a zero object and such that any morphism has a kernel and a cokernel), any morphism $f$ has a canonical factorization $\, f = \coim f \circ \,\bar f\,\circ \,\im f$. We denote by $\Coim f$ and $\IM f$ the source and target of $\bar f$ respectively.

 \medskip Let $\sC$ be an essentially small {\it category with kernels and cokernels} (we do not assume that $\sC$ is additive). Let $\mu$ be a real-valued function on the set of nonzero objects of $\sC$, such that for any nonzero morphism $f$, 
  \begin{equation}\label{mu} \mu(\Coim f)\leq \mu(\IM f).
  \end{equation}
  
  \medskip For any nonzero object $M$, we set 
  $$ \mu_{\max} (M)=  \sup_{N  }\, \mu(N) \;\in \R  \cup \{+ \infty\}$$  
  the supremum being taken over nonzero subobjects $N$ of $M$, or equivalently by \eqref{mu}, over  nonzero kernels of  morphisms with source $M$, and
  
   $$ \mu_{\min} (M)=  \inf_{P  }\, \mu(P) \;\in \R  \cup \{- \infty\}$$  
  the supremum being taken over nonzero quotients $P$ of $M$, or equivalently by \eqref{mu}, over  nonzero cokernels of  morphisms with target $M$,

  \medskip Let us assume in addition that $\sC$ is a {\it monoidal} category with respect to a tensor product $\otimes$ and unit $\un\,$  (we do not assume that $\otimes $ is symmetric).  We also assume the formula
   \begin{equation}\label{mutens} \mu(M_1\otimes M_2)= \mu(M_1)+\mu(M_2).   \end{equation}  
 
   \smallskip Let us assume that $\sC$ is {\it anti-equivalent to itself} via a functor $()^\vee : \;\sC\to \sC^{op}$  which is  related to $\otimes$ via a morphism $\,u\,$ of functors   from $\sC \times \sC$ to sets:  $ \;\sC(\un , M_1\otimes M_2) \overset{u_{M_1,M_2}}{\to} \sC( M_2^\vee, M_1).$ We assume that $u_{M_1,M_2}$ sends nonzero morphisms to nonzero morphisms.
     We also assume the formula
   \begin{equation}\label{mudual} \mu(M^\vee)= - \mu(M). \end{equation} 
   It follows from this, and the fact ${}^\vee$ is an equivalence, that
      \begin{equation}\label{mumin}\, \mu_{\min} (M^\vee) = - \mu_{\max} (M).\end{equation}
   
On the other hand, let us call an object $L$   {\it invertible} if there exists $L^{\otimes (-1)}$ such that $L\otimes L^{\otimes (-1)}  \cong L^{\otimes (-1)} \otimes L\cong  \un$. For such an $L$ and any $M$, the functor $L^{\otimes (-1)} \otimes -$ induces a bijection between isomorphism classes of subobjects of $M$ and of subobjects of $L^{\otimes (-1)} \otimes M $. Therefore
  \begin{equation}\label{muinv}\, \mu_{\max}(L^{\otimes (-1)} \otimes M ) =   \mu_{\max} (M) + \mu(L^{\otimes (-1)}) =  \mu_{\max} (M) - \mu(L).   \end{equation}

   \medskip For any nonzero object $M$, we set 
  $$ \nu (M)=  \sup_{L  }\, \mu(L) \;\in \R  \cup \{\pm \infty\}$$  
  the supremum being taken over invertible subobjects $L$ of $M$ (by convention, this is $-\infty$ is there is no such $L$). Obviously, $ \nu (M)\leq \mu_{\max} (M)$, and for any subobject $N$ of $ M,\; \nu(N) \leq \nu(M)$.

\begin{lemma}\label{lemmanu}  
 \begin{equation}\label{nu}  \nu(M_1\otimes M_2)   \leq   \mu_{\max}(M_1) + \mu_{\max}(M_2). \end{equation}
  \end{lemma} 
  
        \begin{proof}  Any morphism $f: L\to M_1\otimes M_2$ gives rise to a morphism $\un \to L^{\otimes (-1)}\otimes M_1\otimes M_2 $, and in turn to a morphism $f':  M_2^\vee {\to}   L^{\otimes (-1)} \otimes M_1 \,$, which is nonzero if $f\neq 0$.
Applying \eqref{mu} to the canonical factorization of  $f'$, we get
  $\mu_{\min}(M_2^\vee)\leq \mu_{\max}(  L^{\otimes (-1)} \otimes M_1)$, hence, taking into account \eqref{mumin} and \eqref{muinv},  
  $\,\mu(L) \leq  \mu_{\max}(M_1)+ \mu_{\max}(M_2) $, which proves \eqref{nu}.
  \end{proof}
   
        Now, let $\rho$ be another real-valued function on the set of nonzero objects of $\sC$ such that:        
             \begin{equation}\label{rhomon}{\text{ for any subobject}}\; N\,{\text{of}}\; M,\; \rho(N) \leq \rho(M) , \end{equation}
        \begin{equation}\label{rho}   \rho(M_1\otimes M_2)= \rho(M_1)+\rho(M_2),\end{equation}
 \begin{equation}\label{rhonu} \mu(M)\leq \nu(M)+\rho(M).  \end{equation}
 Applying \eqref{rhonu} to subobjects of $M$, and taking \eqref{rhomon} into account, one gets
    \begin{equation}\label{rhonumax}  \mu_{\max}(M)\leq \nu(M)+\rho(M),   \end{equation}
    and one finally derives from  \eqref{nu}, \eqref{rhonumax} (for $M= M_1\otimes M_2$) and \eqref{rho} that
      \begin{equation}\label{mumaxrho}  \mu_{\max}(M_1\otimes M_2)   \leq   (\mu_{\max}+\rho)(M_1) +  ( \mu_{\max}+\rho)(M_2). \end{equation}

\begin{remark} $i)$ Formula \eqref{mutens} (which has been used only in the case when one factor is invertible) implies 
$$   \mu_{\max}(M_1) + \mu_{\max}(M_2)\leq \mu_{\max}(M_1\otimes M_2) ,$$ provided the tensor product of two monomorphisms is a monomorphism (in fact, it suffices that the tensor product of two kernel morphisms is a monomorphism). 

$ii)$ The above conditions on $(\sC, \otimes, {}^\vee, \mu)$ are fulfilled in the case of a {\it quasi-tannakian category} over a field of characteristic zero (\cf \cite[\S\S 7,8]{A}) with a {\it determinantal} slope function $\mu$ (\ie, which satisfies the formula $\mu(M)=\mu({\rm det}\,M)/\rk M$). In this case, $L^{\otimes (-1)}= L^\vee$, and the maps $u_{M_1,m_2}$  are bijections given by composition   
    $   M^\vee_2  \overset{f\otimes 1}{\to} M_1\otimes M_2\otimes  M^\vee_2  \overset{1\otimes \varepsilon_{M_2} }{\to} M_1 ,$ where $\varepsilon_{M_2}: M_2\otimes M_2^\vee \to \un$ is the evaluation morphism.
 \end{remark}

\subsection{}   The above reasoning covers the case of euclidian lattices, with $\,\rho = \frac{1}{2}\log \rk\,$ (\cf subsection 2.1). 
The case of vector bundles requires a slightly more refined setting, to account for finite base changes.

\smallskip Let $I$ be a directed poset and let $(\sC_i, \otimes)_{i\in I}$ be an inductive system of essentially small monoidal categories satisfying the above requirements. We also assume that each $\sC_i$ carries a  (weak right duality) functor  $(\;)^{\vee_i}$, and is endowed with a function $\mu_i $, satisfying the above requirements (no compatibility is required between the $()^{\vee_i}$'s and between the $\mu_i$'s respectively, when $i$ varies).

One defines
$$\tilde \mu =  \limsup_I\, \mu_i,\;\; \tilde \mu_{\max}=  \limsup_I\, \mu_{i,\max},\; \;\tilde \nu=  \limsup_I\,  \nu_i.$$
 Obviously, $  \tilde\nu (M)\leq  \tilde\mu_{\max} (M)$, and for any subobject $\; N\subset M,\;  \tilde\nu(N) \leq  \tilde\nu(M)$. 

Inequality \eqref{nu} applies at each stage $\sC_i$, and gives 

\begin{lemma}  
 \begin{equation}\label{tildenu}  \tilde\nu(M_1\otimes M_2)   \leq   \tilde\mu_{\max}(M_1) + \tilde\mu_{\max}(M_2). \;\;\; \;\;\; \square\end{equation}
  \end{lemma} 

 If $I$ has a minimum, and $\sC$ denotes the category indexed by this minimum, and if $\rho$ is a function as  above, except that \eqref{rhonu} is replaced by
   \begin{equation}\label{rhonutilde} \tilde \mu(M)\leq \tilde\nu(M)+ \rho(M),   \end{equation}
   one finally gets  (by combining the last two inequalities, for $M=M_1\otimes M_2$)   \begin{equation}\label{mumaxrhotilde}  \tilde\mu_{\max}(M_1\otimes M_2)   \leq   (\tilde\mu_{\max}+\rho)(M_1) +  ( \tilde\mu_{\max}+\rho)(M_2). \end{equation}

\subsection{}  This covers the case of vector bundles on a projective smooth curve $S$, by taking $\rho =0$.
 Indeed, fix an algebraic closure $K^{alg}$ of the function field $K$, let $I$ be the set of subfields of $K^{alg}$ containing $K$, ordered by inclusion; let $S_i$ be the normalization of $S$ in the field $K_i$ corresponding to $i$, and $\sC_i$ be the category of vector bundles on $S_i$.
 Define $\mu_i= \mu/[K_i:K]$. 
  Then 
  $\,\mu(E) = {\tilde\mu}(E),$ and {\it $E$ is {nef} if and only if $\tilde\nu(E^\vee)\leq 0$}. By Kleiman's theorem,
 $\,{\tilde\mu}(E) \leq \tilde\nu(E)$, which actually implies  \begin{equation}\label{Kleimanu}\,{\tilde\mu}_{\max}(E) = \tilde\nu(E).\end{equation}
One has $\, \mu_{\max}(E)\leq {\tilde\mu}_{\max}(E),$ and this is an equality in characteristic $0$,  by Galois descent of the Harder-Narasimhan slope filtration (this is the only place where this filtration is used).

 From \eqref{Kleimanu},  the equality $ \tilde\nu(E_1\otimes  E_2)  =   \tilde\nu(  E_1) + \tilde\nu( E_2)$ appears as a consequence of \eqref{mumaxrhotilde}. 
 This clarifies why ``{nef} $\otimes$ {nef} is {nef}"  in any characteristic (\cf \cite{Ba}\cite{Br}), while ``semistable $\otimes $ semistable is semistable" only in characteristic $0$. Recall that in positive characteristic, a vector bundle $E$ is called {\it strongly semistable} if $\mu(E)= \tilde\mu_{\max}{(E)}$(this amounts to requiring that any iterated Frobenius pull-back is semistable). It follows immediately from \eqref{mumaxrhotilde} that  ``strongly semistable $\otimes $ strongly semistable is strongly semistable" (\cf \cite[cor. 7.3]{Mo0}.

\subsection{}  
This applies in a similar way to hermitian ${\frak o}_K$-lattices, taking $\rho= \frac{[K:\Q]}{2}\log\rk$. One has $\mu(\bar E) = {\tilde\mu}(\bar E),\;  \mu_{\max}(\bar E)={\tilde\mu}_{\max}(\bar E) ,$ and $ \,{\tilde\mu}(\bar E) \leq \tilde\nu(\bar E)+ \rho(\bar E)$ by Lemma \ref{zhan} (consequence of Zhang's theorem). 
 
 On the other hand, we have seen that it is not true that $ \tilde\nu(\bar E_1\otimes \bar E_2)   \leq   \tilde\nu(\bar E_1) + \tilde\nu(\bar E_2)$ in general.
 
 \subsection{}  The analog of Theorem 0.5 for Higgs bundles ${\mathfrak E}= (E, \;\theta: E\to E\otimes \omega_S)$ on a smooth projective curve $S$ in characteristic $0$ is known, \cf \cite[Cor. 3.8]{Si} (the slope of ${\mathfrak E}$ is the slope of $E$, and  $\mu_{max}({\mathfrak E})$ is the supremum of slopes of Higgs subbundles). One could ask whether the above strategy applies in this context.  However, there are Higgs bundles of rank three with nilpotent $\theta$ for which $\tilde \mu({\mathfrak E})> \tilde\nu({\mathfrak E})$ , \cf  \cite[3.4]{BH}.
 
\section{Generalized vector bundles}

\subsection{} After having given the most general framework where our argument works, we consider a quite concrete categorical context which contains both contexts of vector bundle and hermitian lattices. Namely, we introduce the notion of generalized vector bundle, following ideas in \cite{Gau} and \cite{HoJS}.

Let $K$ be a field endowed with a collection $(\vert\;\vert_v)_{v\in V}$ of (not necessarily distinct) absolute values satisfying the product formula: for every $a\in K\setminus \{0\}$, $\vert a\vert_v = 1$ for all but finitely many $v$, and $\prod \vert a\vert_v = 1$. We denote by $K_v$ the completion of $K$ at $v$. If $v$ is archimedean, $K_v$ is isomorphic to $\R$ or $\C$, and one assumes that $\vert\;\vert_v$ coincides with the standard absolute value.

\smallskip A {\it generalized vector bundle} $\bar M = (M,(\vert\vert \; \vert\vert_v)_{v\in V})$  over $K$  is the data of a finite-dimensional $K$-vector space together with a $\vert\;\vert_v$-norm $\vert\vert \; \vert\vert_v$ on $M_v := M\otimes_K K_v$ for each $v$. One requires that for every $m\in M\setminus \{0\}$, $\vert\vert m\vert\vert_v = 1$ for all but finitely many $v$. If $\vert\;\vert_v$ is archimedean, one requires that $\vert\vert \; \vert\vert_v$ is euclidean/hermitian; if $\vert\;\vert_v$ is non archimedean, one requires that $\vert\vert \; \vert\vert_v$ is a ultranorm such that the ${\frak o}_{K_v}$-module $\{m\in M_v; \vert\vert m \vert\vert_v\leq 1\}$ is an ${\frak o}_{K_v}$-lattice of $M_v$ (so that $M_v$ admits an orthonormal basis).

\smallskip A {\it morphism} of generalized vector bundles is a $K$-linear map which is of norm $\leq 1$ on each $v$-completion.

\smallskip Generalized vector bundles form a category with kernels and cokernels, the norms being the induced and quotient norms respectively (this category is additive if and only if all the $\vert\,\vert_v$ are non-archimedean). Moreover, our assumption on the norms $\vert\vert \; \vert\vert_v$ allow us to define the tensor product in a standard way (for each $v$, the tensor product of orthonormal bases is orthonormal), so that the category of generalized vector bundles becomes monoidal symmetric. One also defines ``duals" in a standard way. 

\smallskip  There is a natural notion of determinant $\det \bar M$: as vector space, this is the top exterior power, and for every $v$, the determinant of any orthonormal basis has norm $1$ (note that $\det \bar M$ is not a quotient of the corresponding tensor power if there is some archimedean $\vert \; \vert_v$).

If $\rk \bar L = 1$, with generator $\ell$, one sets
\begin{equation}\mu(\bar L)= -\sum \log \vert\vert \ell\vert\vert_v.\end{equation}
By the product formula, this does not depend on $\ell$. In general one sets
\begin{equation}\mu(\bar M)=  \frac{\mu(\det \bar M)}{\dim  M}.  \end{equation}
It is easy to check conditions \eqref{mu}\eqref{mutens}\eqref{mudual}.

\smallskip { Lemma \ref{lemmanu} then applies to generalized vector bundles}.
  In addition, the product $\mu\cdot \dim$ behaves like a degree: it is additive with respect to short exact sequences of multifiltered spaces (this follows from its definition in terms of the determinant).

\subsection{Examples} $1)$ If $K$ is a number field, and $V$ is the set of its places $v$ (counted with multiplicity $[K_v:\Q_{p(v)}]$), then the category of generalized vector bundles is equivalent to the category of hermitian ${\frak o}_K$-lattices (the point is that $M_{{\frak o}_K}:= \bigcap_v\, (M\cap  \{m\in M_v, \vert\vert m \vert\vert_v\leq 1\})$ is an ${\frak o}_{K}$-lattice of $M$. Indeed, this is clear in rank one. In general, our conditions on the norm clearly imply that $M_{{\frak o}_K}$ contains a lattice, and any non-decreasing sequence of lattices contained in $M_{{\frak o}_K}$ has to stabilize since so do their determinants). The functions $\mu$ coincide.

\smallskip\noindent $2)$ If $K$ is the function field of a projective smooth curve $S$ over a field $k$, then the category of generalized vector bundles is equivalent to the category of vector bundles on $S$.  The functions $\mu$ coincide.
 
 \smallskip \noindent $3)$ If $K$ is any field, $V$ is finite and all $\vert\;\vert_v$ are trivial,  then the category of generalized vector bundles is equivalent to the category of finite-dimensional $K$-vector spaces with a finite collection (indexed by $V$) of decreasing exhaustive separated left-continuous\footnote{\ie such that $F_v^{\geq\lambda}M  = \bigcap_{\lambda'<\lambda}\,F_v^{\geq\lambda}M $.} filtrations indexed by $\R$: define 
 $$F_v^{\geq\lambda}M = \{m\in M ; \vert\vert m \vert\vert_v \leq e^{-\lambda}\}.$$
 
 The function $\mu$ coincides with
 the one considered by G. Faltings \cite{F}: 
 $$\mu(\bar M) = \frac{1}{\dim M}\sum_{v, \lambda}   \, \lambda \dim_K gr_{F_v}^\lambda \bar M.$$

\section{Tensor product of semistable multifiltered spaces (proof of Theorem 0.7)}

\subsection{} It is enough to take $i\in \{1,2\}$ and to establish the inequality
\begin{equation}\label{oncemore}\mu_{\max}(\bar M_1\otimes \bar M_2)\leq \mu_{\max}(\bar M_1)+ \mu_{\max}(\bar M_2),\end{equation}  using  the strategy of subsection 5.1.

What remains to prove is inequality \eqref{rhonu} with $\,\rho=0$, that is:
 \begin{equation}\label{zeronu} \mu(\bar M)\leq \nu(\bar M) .  \end{equation}
 
We may assume $V=\{1, 2,\ldots, n\}$. We shall first prove 

 \begin{lemma}  $ \mu(\bar M) = \frac{1}{\dim M}\sum_{ \lambda_1,\ldots, \lambda_n}   \, ( \lambda_1+\cdots +\lambda_n) \dim_K gr_{F_1}^{\lambda_1} \ldots gr_{F_n}^{\lambda_n} \bar M $ 

\smallskip\noindent  (where, for $m<n,\, gr_{F_{m+1}}^{\lambda_n} \ldots gr_{F_n}^{\lambda_1} \bar M$ is given the $m$ filtrations induced by  $F^{\geq .}_{ 1},\ldots , F^{\geq .}_m $ in that order). 
  \end{lemma}

\begin{proof} Since the product $\mu\cdot \dim$ is additive with respect to short exact sequences of multifiltered spaces (using induced and quotient filtrations), this allows, by descending induction on lexicographically ordered pairs $(v,\lambda)$, to replace $\bar M$ by $\oplus_{ \lambda_1,\ldots, \lambda_n}     gr_{F_1}^{\lambda_1} \ldots gr_{F_n}^{\lambda_n} \bar M $ in the formula, each term $ gr_{F_1}^{\lambda_1} \ldots gr_{F_n}^{\lambda_n} \bar M$ being considered as multifiltered with only one notch $\lambda_v$ for the filtration $F^{\geq .}_v$.  
\end{proof}

Let then ${ \lambda_1,\ldots, \lambda_n}$ be such that $ \lambda_1+\cdots +\lambda_n$ is maximal (say with value  $\lambda$)  with $\, gr_{F_1}^{\lambda_1} \ldots gr_{F_n}^{\lambda_n} \bar M \neq 0$. Let $m\in F_1^{\geq \lambda_1}\ldots F_n^{\geq \lambda_n} M$ lift some nonzero vector in $ gr_{F_1}^{\lambda_1} \ldots gr_{F_n}^{\lambda_n} \bar M$, and let $\bar L$ be the subobject of $\bar M$ of rank one generated by $m$.
Then it is clear from the lemma that $\mu(\bar M)\leq \lambda = \mu(\bar L).$ Therefore $\mu(\bar M)\leq \nu(\bar M)$. \qed

\subsection{} Instead of multiple filtrations on a $K$-vector space, M. Rapoport \cite{Ra} considers, for a finite separable extension $L/K$,  the data $\bar M$ of a $K$-vector space $M$ endowed  with one filtration on $M_{L}= M\otimes_K L$, and  defines
$$\mu(\bar M) = \frac{1}{\dim M}\sum_{  \lambda}   \, \lambda \dim_{L} gr^\lambda M_{L}.$$
 
 \medskip The two settings can be unified by allowing more generally $L$ to be a finite etale $K$-algebra (\ie, a finite product of finite separable extensions of $K$): if $L$ is a product of $n$ copies of $K$, a filtration on $M_{L}$ amounts to the data of $n$ filtrations on $M$.
 
Let us generalize \eqref{oncemore} to this more general context.
 Let $ K^{sep}$ be a fixed separable closure of $K$. Then for any finite etale $K$-algebra $L$, there exists a finite Galois extension $K'/K$ in $  K^{sep}$ such that $L\otimes_K K'$ is a product of copies of $K'$, and one recovers the case of multifiltrations. We set $\bar M'= \bar M\otimes_K K'$.
 It is clear that $\mu(\bar M)= \mu(\bar M')$, and one also has
 $\mu_{\max}(\bar M)= \mu_{\max}(\bar M')$ by Galois descent of the (unique) subobject of $M'$ of maximal rank with maximal slope.

  \end{sloppypar}

\bigskip
\address{ 
D\'epartement de Math\'ematiques et Applications  \\ 
\'Ecole Normale Sup\'erieure\\
45 rue d'Ulm,  75230
  Paris Cedex 05\\France  \\
 }
{andre@dma.ens.fr}

\end{document}